\documentclass{article}
\usepackage[utf8]{inputenc}
\usepackage{titlesec}	
\usepackage[english]{babel}
\usepackage{pst-plot}
\usepackage{nicematrix}
\usepackage{caption}
\usepackage{authblk}
\usepackage{subcaption}
\usepackage{amsthm,amssymb,amsmath,amsfonts,calrsfs,upgreek}
\usepackage[mathscr]{euscript}
\usepackage{mathrsfs}

\newtheorem{introthm}{Theorem}

\newtheorem{mythm}{Theorem}[section]
\newtheorem{myexa}{Example}[section]
\newtheorem{mydef}{Definition}[section]

\newtheorem{mycor}{Corollary}[section]
\newtheorem{mypro}{Proposition}[section]

\newtheorem{myconv}{Convention}[section]
\newtheorem{myrem}{Remark}[section]
\newtheorem{myque}{Question}[section]
\newcommand{\D}{\mathcal{D}}

\newcommand{\B}{\mathcal{B}}

\newcommand{\e}{\epsilon}
\newcommand{\ee}{\varepsilon}

\usepackage{tikz}
\usetikzlibrary{automata, positioning, arrows, decorations.pathreplacing,angles,quotes}
\tikzset{
    ->, 
    >=stealth, 
    node distance=3cm, 
    every state/.style={thick, fill=gray!10}, 
    every edge/.append style={line width=0.25mm}, 
    initial text=$ $, 
}
\title{The cogrowth inequality from Whitehead's algorithm}
\author{Asif Shaikh}
\affil[]{asif.shaikh@fulbrightmail.org}
\date{May 4, 2025}
\usepackage{graphicx}
\begin{document}
\maketitle
\begin{abstract}
This article focuses on non-cyclic free factors $H\leq F_m$ of the free group $F_m$ with finite rank $m > 2$, and specifically addresses the implications of Ascari's refinement of the Whitehead automorphism $\varphi$ for $H$ as introduced in \cite{ascari2021fine}. Ascari showed that if the core $\Delta_H$ of $H$ has more than one vertex, then the core $\Delta_{\varphi(H)}$ of $\varphi(H)$ can be derived from $\Delta_H$. We consider the regular language $L_H$ of reduced words from $F_m$ representing elements of $H$, and employ the construction of $\B_H$ described  in \cite{DGS2021}. $\B_H$ is a finite ergodic, deterministic automaton that recognizes $L_H$. Extending Ascari's result, we show that for the aforementioned free factors $H$ of $F_m$, the automaton $\B_{\varphi(H)}$ can be obtained from $\B_H$. Further, we present a method for deriving the adjacency matrix of the transition graph of $\B_{\varphi(H)}$ from that of $\B_H$ and establish that $\alpha_H < \alpha_{\varphi(H)}$, where $\alpha_H, \alpha_{\varphi(H)}$ represent the cogrowths of $H$ and $\varphi(H)$, respectively, with respect to a fixed basis $X$ of $F_m$. The proof is based on the Perron-Frobenius theory for non-negative matrices.
\end{abstract}
{\bf Keywords:}  Cogrowth,  Regular language, Free group, Whitehead algorithm, Free factors\\
{\bf Mathematics Subject Classification – MSC2020:} 20E05, 20F69, 20F65
\section{Introduction}\label{sec:intro}

The automorphism problem for a free group $F_m$, where $m \geq 2$, addresses the question of whether there exists an automorphism $\varphi : F_m \rightarrow F_m$ such that $\varphi(w) = w'$, for any given pair of elements $w$ and $w' \in F_m$. In a seminal paper \cite{WhiteheadMR1503309}, J.H.C. Whitehead presented an algorithm to solve this problem. An element (or word) $w \in F_m$ is primitive if it belongs to some basis of $F_m$. Whitehead's algorithm is based on the below theorem. For a proof of the theorem below, the reader is referred to \cite{WhiteheadMR1503309}.

\begin{introthm}[Whitehead]\label{introWhitehead}
Let $w$ be a cyclically reduced word, which is primitive and not a single letter. Then there is a Whitehead automorphism $\varphi$ such that the cyclic length  of $\varphi(w)$ is strictly smaller than the cyclic length of $w$.
\end{introthm}

This theorem was originally established by Whitehead using topological techniques involving three-manifolds. In a later development, Rapaport \cite{Rapaport} provided a purely algebraic proof, her approach was subsequently refined by Higgins and Lyndon \cite{HL1974}. These algebraic approaches not only confirmed Whitehead’s result but also broadened its accessibility within combinatorial group theory. This result addresses the primitivity problem, as a word w is primitive if and only if the minimal length of its orbit under automorphisms is 1.

The idea of extending Whitehead's minimization framework to finitely generated subgroups by minimizing the number of vertices in the core graph was proposed by Gersten \cite{Gersten1984}, and was rigorously developed and proven by Kalajdzievski \cite{Saso1992}. Further, Roig, Ventura, and Weil \cite{RVW2007} provided a polynomial-time algorithm through a graphical reformulation and analysis of cut-capacity in Whitehead hypergraphs. The subgroup analog addresses the free factor problem, as a subgroup $H$ is a free factor if and only if the minimum size of an element in its automorphic orbit is 1.

In a notable refinement, Ascari \cite{ascari2021fine} constructed an algorithm that ensures the reduction of the size of the core graph by collapsing certain edges. Let $\Delta_H$ be the core graph of $H \leq F_m$. It is the minimal finite, connected, and folded graph that contains all reduced loops corresponding to words in $H$, capturing the essential structure of $H$ without any redundant paths (see Section \ref{sec:Schreier_core} for the formal definition).  This result is formalized in the following theorem:

\begin{introthm}\label{introascari} \textnormal{(Theorem \ref{Whitehead2}) }
Let $H \leq F_m$ be a free factor and suppose $\Delta_H$ has more than one vertex. Then there is a Whitehead automorphism $\varphi$ such that $\Delta_{\varphi(H)}$ has strictly fewer vertices and strictly fewer edges than $\Delta_H$. Additionally, the automorphism $\varphi$ can be chosen in such a way that $\Delta_{\varphi(H)}$ can be obtained from $\Delta_H$ by collapsing certain edges of $\Delta_H$.
\end{introthm}

The significance of this result lies in its production of the simplified core graph for the image subgroup, achieved by reducing both the number of vertices and edges. This controlled reduction opens up intriguing possibilities for studying the dynamical properties of the subgroup, particularly in terms of its cogrowth and entropy.

An important aspect of subgroup dynamics in free groups is the study of cogrowth, which refers to the growth rate of the number of reduced words in a subgroup relative to the group. The relationship between cogrowth and the spectral radius of a simple random walk on the Schreier graph of a subgroup, as first introduced by Grigorchuk \cite{MR599539Gri1980}, has become a fundamental tool for understanding the asymptotic and probabilistic properties of the subgroup via its Schreier graph. In this context, minimizing the size of the core of subgroup can offer valuable insights into both its cogrowth and the related entropy, which reflect the subgroup's structural properties. This motivates the results discussed in this paper.

We first extend Ascari’s result to the ergodic automaton $\B_H$ introduced in \cite{DGS2021}, which recognizes the language $L_H$ of reduced words in $F_m$ that represent elements of $H$. While \cite{DGS2021} uses the notation $\widehat{\D}_H$ for this automaton, we adopt the simplified notation $\B_H$. Ascari considers the case where $H$ is a free factor of $F_m$, but for our purposes, we require $H$ to be a non-cyclic free factor, i.e., $rk(H) > 1$. This ensures that the automaton $\B_H$ is connected and the cogrowth of $H$ is strictly greater than 1, both of which are essential for our spectral comparison. We next describe how the automaton $\B_H$ associated with $H$ changes when we apply $\varphi$ to $H$.

\begin{introthm}\label{introshai1} \textnormal{(Theorem \ref{thm:B_from_BH})}
Let $\varphi$ be the automorphism given by Theorem \ref{introascari}. Then the ergodic automaton $\B_{\varphi(H)}$ that recognizes $L_{\varphi(H)}$ can be obtained from $\B_H$ by collapsing certain edges of $\B_H$.
\end{introthm}

Let $M = M_H$ and $M_1 = M_{\varphi(H)}$ denote the adjacency matrices of the transition diagrams corresponding to $\B_H$ and $\B_{\varphi(H)}$, respectively. The irreducibility of $M$ and $M_1$ follows from the ergodicity of $\B_H$ and $\B_{\varphi(H)}$. Denote by $\lambda = \lambda_H$ and $\lambda_1 = \lambda_{\varphi(H)}$ the Perron-Frobenius eigenvalues of $M$ and $M_1$, respectively. We derive the following cogrowth inequality: 

\begin{introthm}\label{introshai2} \textnormal{(Theorem \ref{thm:main_lambda<lambda1})}
Let $\varphi$ be the automorphism given by Theorem \ref{introascari}. Then $\lambda < \lambda_1.$ 
\end{introthm}

This inequality suggests that the Whitehead techniques, when extended to the context of automata, reveal structural and dynamical changes in the subgroup. By moving from core graphs to ergodic automata, we gain a more refined, language-theoretic perspective on these phenomena. Hence, it is natural to ask whether the automorphisms arising from Ascari’s method, or those obtained via cut-capacity arguments, yield the maximal cogrowth representative within the Whitehead automorphic orbit of a given subgroup $H \leq F_m$. More generally, 
given a finitely generated subgroup $H \leq F_m$, can one determine the Whitehead automorphism $\varphi$ of $F_m$ such that the cogrowth $\alpha_{\varphi(H)}$ is maximized over the Whitehead automorphic orbit of $H$? Is there an effective method or algorithm to find such a Whitehead automorphism? We refer to this as the Whitehead maximal cogrowth problem. To the best of our knowledge, this question remains open.

We conclude the introduction by outlining the contents of the rest of the paper. Section 2 revisits the fundamental definitions and terminologies from the theory of automata. Additionally, it presents the construction of the ergodic automaton $\B_H$ that recognizes $L_H$, as introduced in \cite{DGS2021}. In Section 3, we review Ascari's refinement of the Whitehead algorithm for subgroups, as discussed in \cite{ascari2021fine}. The proofs of Theorems \ref{introshai1} and \ref{introshai2} are completed in Section 4. Finally, we conclude with an illustrative example and a section outlining several open problems for future research.

\section{The automaton $\B_H$}
The main objective of this section is to revisit the definition of $\B_H$ from \cite{DGS2021} and discuss the key properties of the automaton $\B_H$. 
\subsection{Preliminaries}
We adopt the following conventions and terminologies. We use $X=\{x_1, \ldots, x_m\}$ to denote a fixed basis of the free group $F_m$, where the elements of $X$ and their inverses are treated as formal letters within the context of formal languages. Thus, the set of generators $\Sigma = X \cup X^{-1}$ of $F_m$ is considered as an alphabet when discussing formal languages.

We denote the set of all finite words over the alphabet $\Sigma$ as $\Sigma^*$. From an algebraic perspective, $\Sigma^*$ represents the free monoid generated by the finite set $\Sigma$. The $length$ of a word $w \in \Sigma^*$ is denoted by $|w|$ and refers to the number of letters in $w$, counting each letter as many times as it appears. Subsets of $\Sigma^*$ are commonly referred to as (formal) $languages$ over the alphabet $\Sigma$. A language $L$ is called as $regular$, if it is recognized by a finite automaton. A finite automaton $\B$ is defined as a quintuple $\B = (Q,\Sigma,\delta,I,F)$, comprising a finite set of states $Q$, an alphabet $\Sigma$, a transition function $\delta:Q\times\Sigma \rightarrow 2^Q$, a set of initial states $I \subseteq Q$, and a set of final states $F \subseteq Q$.

Let $G_{\B}$ or simply $\B$ be the transition diagram of $\B$, that is $\B$ is a labeled directed graph with vertex set $Q$ and the directed labeled edges are described by the transition function $\delta$ with labels from $\Sigma$. Namely, vertex $q$ is connected with vertex $q'$ with an edge labeled by $x \in \Sigma$, if $q' \in \delta(q, x)$. (For example, Figure \ref{fig:B_H} is a depiction of a transition diagram for an automaton.) Let $e$ be an edge in $\B$. We use $o(e), t(e)$ and $\mu(e)$ to denote origin, terminus and the label of the edge $e$, respectively. In the context of $\B$, a directed path $p=e_1\cdots e_n$ in $\B$ is called \emph{admissible} if $o(e_1) \in I$, $t(e_i) = o(e_{i+1}),$ for $i = 1,\cdots,(n-1)$, $t(e_n) \in F$. Let $w = y_1\cdots y_n$ be a word over $\Sigma$. The automaton $\B$ accepts the word $w$ if there is an admissible path $p$ in $G$ such that $\mu(p) = \mu(e_1)\cdots \mu(e_n) = w$. The language recognized by $\B$, denoted as $L(\B)$, is the set of words accepted by $\B$.

An automaton $\B$ is \emph{ergodic} if its transition diagram is strongly connected, that is, for any two states $q$ and $q' \in Q$ there exists a path connecting $q$ to $q'$. A language $L \subseteq \Sigma^*$ is \emph{irreducible} if, given two words $w_1, w_2 \in L$, there exists a word $w \in \Sigma^*$ such that the concatenation $w_1ww_2$ belongs to $L$. A regular language $L$ is irreducible if and only if it is recognized by some ergodic automaton, see Theorem 3.3.11 of \cite{Lind-Marcus}. An automaton $\B$ is \emph{unambiguous} if for every $w \in L(\B)$, there is a unique admissible path $p \in \B$ such that $\mu(p) = w$. An automaton $\B$ is \emph{deterministic}, if for each state of $Q$, all outgoing edges carry distinct labels. It is obvious that a deterministic automaton with one initial state is unambiguous. Note that $\B$ is deterministic if the codomain of $\delta$ is $\{\emptyset\} \cup Q$, that is $\delta:Q\times\Sigma \rightarrow \{\emptyset\} \cup Q$. Let $k \geq 1$. An automaton $\B$ has \emph{homogeneous ambiguity} $k$ if, for any nonempty word $w \in L(\B)$, there are exactly $k$ admissible paths $p_1,\cdots,p_k$ in $\B$ with label $w$. The $in$ degree of a vertex $v$ of the directed graph $G$ is the number of edges in the graph that have $v$ as the terminus. Similarly, the $out$ degree of a vertex $v$ of the graph is the number of edges in the graph that have $v$ as the origin. 

\subsection{The Schireier and core graph of $H$}\label{sec:Schreier_core}
The construction of $\B_H$ relies on the core $\Delta_H$ of a Schreier graph $\Gamma$ associated to $H$. We define two versions of the Schreier graph associated with $H \leq F_m$, which we denote by $\Gamma$ and $\widehat{\Gamma}$, respectively. The set of vertices of $\Gamma$ and $\widehat{\Gamma}$ is the same and is the set $V = \{ Hg \mid g \in F_m\}$ of right cosets. The set of edges $E$ of $\Gamma$ is the set $E = \{(Hg, Hgx)\mid  g \in F_m, x \in X\}$ consisting of pairs $e = (Hg, Hgx)$ of cosets. The edges are oriented and $Hg$ is the origin $o(e)$ of $e$ while $Hgx$ is the terminus $t(e)$ of $e$. Moreover, such an edge has the label $\mu(e) = x.$ Each vertex in $\Gamma$ has $m$ outgoing edges whose labels constitute the set $X$. The graph $\widehat{\Gamma}$ is obtained from $\Gamma$ by adding edges from the set $\overline{E} = \{\overline{e}\mid e \in E\}$ where $\overline{e} = (Hgx,Hg)$ if $e = (Hg, Hgx)$ and the label $\mu(\overline{e}) = \mu(e)^{-1} = x^{-1} \in X^{-1}.$ Thus $\Gamma = (V, E, \mu)$ and $\widehat{\Gamma} = (V,E\cup \overline{E},\widehat{\mu})$, where $\widehat{\mu}(e) = \mu(e)$ if $e \in E$ and $\widehat{\mu}(\bar e) = \mu(e)^{-1}$ if $\bar e \in \overline{E}$. Each vertex of $\widehat{\Gamma}$ has $2m$ outgoing edges and $2m$ incoming edges, whose labels constitute the set $\Sigma = X \cup X^{-1}$. We call $\Gamma$ the \emph{Schreier graph} and $\widehat{\Gamma}$ the \emph{extended Schreier graph} of $H$. The vertex $v_1 = H$ is the distinguished vertex, so in fact $\Gamma$ and $\widehat{\Gamma}$ are rooted graphs with root $v_1$. 

The \emph{core} $\Delta_H = (\widehat{V},E_{\Delta_H},\mu)$ is the subgraph of the Schreier graph $\Gamma$ that is defined as the union of closed paths containing the root vertex $v_1$. Thus, since the Schreier graph $\Gamma$ is connected, its core $\Delta_H$ is also connected. Let $\overline{E}_{\Delta_H} = \{\overline{e}\mid e \in E_{\Delta_H}\}$. We now define the {\it extended core} graph $\widehat{\Delta}_H = (\widehat{V},\widehat{E},\widehat{\mu})$ from the core $\Delta_H$, where $\widehat{E} = E_{\Delta_H}\cup\overline{E}_{\Delta_H}$. Observe that the extended core $\widehat{\Delta}_H$ is a subgraph of $\widehat{\Gamma}.$ We say that a labeled path  is \emph{reduced} if it does not contain adjacent edges with labels of the form $aa^{-1}$, otherwise, we say that the path is not reduced or  we say that it \emph{backtracks}. Note that paths in the graph $\widehat{\Delta}_H$ are not necessarily reduced and may backtrack. For example, a path $p = e\overline{e}$ in $\widehat{\Delta}_H$ from $v_1$ to $v_1$, where $e, \overline{e} \in \widehat{E}$ and $o(e) = v_1 = t(\overline{e}), t(e) = v' = o(\overline{e})$ is not a reduced path. It is possible to view $\widehat{\Delta}_H$ as a finite, deterministic, unambiguous automaton, with the distinguished vertex $v_1$ is serving as both the unique initial and terminal state. Let $L(\widehat{\Delta}_H)$ denote the set of all words $w = \widehat{\mu}(p)$, where $p$ is a path of $\widehat{\Delta}_H$ that begins and ends at $v_1$, that is, $L(\widehat{\Delta}_H)$ is the language recognized by the automaton $\widehat{\Delta}_H.$ Notice that these paths $p$ in $\widehat{\Delta}_H$ may or may not be reduced. Hence, not all words in the language $L(\widehat{\Delta}_H)$ are reduced. We denote by $L_H$ the language of reduced elements of a f.g. subgroup $H$ of $F_m.$ Notice that $L_H \subset L(\widehat{\Delta}_H).$ An element $w = y_1\cdots y_n \in F_m$ is called \emph{cyclically reduced} if $y_n \neq y_1^{-1}.$

{\myexa \label{example:H} Let $H = \langle yx^{-1}, yzy^{-1}zt\rangle$ be a subgroup of a free group $F_4 = F\langle x,y,z,t \rangle$. See Figure \ref{fig:coreH} for the core $\Delta_H$.}
\begin{figure}[!htb]
     \begin{subfigure}[b]{0.5\textwidth}
        \centering
         \includegraphics[width=.8\textwidth]{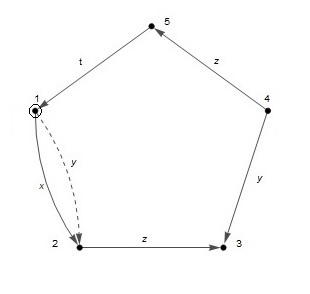}
          \caption{The core $\Delta_H$.}  
        \label{fig:coreH}
       \end{subfigure}
       \hfill
       \begin{subfigure}[b]{0.5\textwidth}
        \centering
           \includegraphics[width=.8\textwidth]{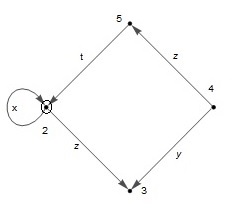}
          \caption{The core $\Delta_{\varphi(H)}$}
        \label{fig:corePhiH}
        \end{subfigure}
        \caption{The circled vertices in Figures \ref{fig:coreH} and \ref{fig:corePhiH} indicate the root vertex.}
        \label{fig:cores}
\end{figure}
{\myconv \label{con1} From now on, we will assume that $H \leq F_m$ is a non-trivial, non-cyclic finitely generated subgroup generated by $w_1,\cdots,w_k$, where $w_i, ~\forall i$ are cyclically reduced words over $\Sigma$, where $\Sigma = X \cup X^{-1}$, $X=\{x_1, x_2, \ldots, x_m\}$ and elements of the set $X$ are generators of $F_m$.}  

\subsection{Definition of $\B_H$}\label{sec:def_B_H}
We are interested in the automaton $\B_H$ because it is ergodic and recognizes the language $L_H$. Also, it provides a tool for calculating the cogrowth of $H$. Following \cite{DGS2021}, we present the definition of $\B_H$ as follows:
$$\B_H = \bigl(Q_H, \Sigma, \delta_H, I_H, F_H\bigr),$$
where 
\begin{eqnarray}
    Q_H = \left\{(v,x_i^{\e}) \mid v \in \widehat{V}, x_i^{\e} \in \Sigma, ~and~ \exists e \in \widehat{E} ~s.t.~ x_i^{\e} = \widehat{\mu}(e),t(e)=v\right\}, \label{eqn:B_H_vertex_set} \\
    I_H = F_H = \Bigl\{(v_1,x_i^{\e})\mid x_i^{\e} = \widehat{\mu}(e),t(e)=v_1\Bigr\}, \label{eqn:B_H_initial_set}\\
    \delta_H\left((v,x_i^{\e}),x_j^{\e'}\right) = \left( \delta_{\widehat{\Delta}_H}(v,x_j^{\e'}),x_j^{\e'}\right)= (vx_j^{\e'},x_j^{\e'}), \textrm{ if } x_i^{\e} \neq (x_j^{\e'})^{-1}. \label{eqn:B_H_transition}
\end{eqnarray}
We now list some of the important properties of $\B_H$ described in \cite{DGS2021}. 
{\mythm \label{thm:B_m_properties}  Let $H \leq F_m$ be as mentioned in the Convention \ref{con1}. Then the automaton $\B_H$ is minimal, ergodic, deterministic and has homogeneous ambiguity $\deg(v_1) - 1$.}
\begin{proof}
   Minimality follows from Theorem 5.8 and Lemma 5.11, ergodicity from Theorem 5.14, and the remaining properties from Proposition 5.12 of \cite{DGS2021}.
\end{proof}
\subsection{Entropy of $L_H$}\label{section-on-entropy} 
The entropy $ent(L)$ of a formal language $L$ is defined as 
\begin{equation}
     ent(L) =  \limsup_{n \rightarrow \infty} {\frac{\log(b_n)}{n}}, \label{eqn:entropy_def}
\end{equation}where $b_n = \#\left\{w \in L ~|~ |w| = n\right\}.$

According to Theorem \ref{thm:B_m_properties}, it follows that the adjacency matrix $M$ of the transition diagram of the deterministic automaton $\B_H$ is non-negative, integral and irreducible. This crucial observation allows for the application of the Perron-Frobenius theory, leading to a theorem concerning the entropy of $L_H$. For a comprehensive understanding of Perron-Frobenius theory, a detailed discussion can be found in Chapter 4 of \cite{Lind-Marcus}. We now recall the Theorem 5.18 from \cite{DGS2021}. 
{\mythm \label{thm-on-computing-entropy}
Let $H \leq F_m$ be as mentioned in the Convention \ref{con1}. Then $ent(L_H) = \log \lambda_H,$ where $\lambda_H$ is the maximal (also called Perrron Frobenius) eigenvalue of the adjacency matrix $M$ of $\B_H$.}

\subsection{Cogrowth of $H$}
Let $H \leq F_m$ be a f.g. subgroup. Let $a_n$ be the number of reduced elements of length $n$ in $H$ with respect to a fixed basis $X$ of $F_m$. The upper limit
\begin{equation}
    \alpha_H = \displaystyle \limsup_{n \rightarrow \infty} \sqrt[n]{a_n} \label{eqn:rate_cogrowth_def}
\end{equation} is called the \emph{cogrowth} of $H$ with respect to a fixed basis $X$ of $F_m$. The term cogrowth was introduced by Grigorchuk in his foundational work \cite{MR599539Gri1980} to distinguish the growth of a subgroup $H \leq F_m$ as measured within the ambient group $F_m$. This contrasts with the intrinsic growth of a group or subgroup viewed on its own. In the literature, this notion is also referred to as \emph{relative growth}, as it quantifies how the subgroup sits inside the larger group. The prefix `co'- reflects this external or comparative viewpoint.
 
Let $L_H = L(\B_H)$ be the language accepted by $\B_H$ constructed in the previous Section. Let $M$ be the adjacency matrix of the transition diagram of $\B_H$. 
Notice from Equation \eqref{eqn:entropy_def}, \eqref{eqn:rate_cogrowth_def} and Theorem \ref{thm-on-computing-entropy} that 
\begin{equation}
  \alpha_H = e^{ent(L_H)} = \lambda_H. \label{eqn:cogrowth_entropy}  
\end{equation}

\section{Whitehead’s algorithm}\label{sec:ascari}
In this section, we focus on the results obtained in \cite{ascari2021fine} regarding the refinement of Whitehead's algorithm. 
\subsection{Whitehead automorphism}
We begin our discussion by providing the definitions of the Whitehead automorphism of $F_m$ and the Whitehead graph. 
\begin{mydef}
Let $a\in\Sigma = X\cup X^{-1}$ and let $A\subseteq\Sigma\setminus\{a,a^{-1}\}$. The \emph{Whitehead automorphism} $\varphi=(A,a)$ is given by $a\mapsto a$ and
\begin{center}
$\begin{cases}
x_j\mapsto x_j & \text{if } x_j,x_j^{-1}\not\in A,\\
x_j\mapsto ax_j & \text{if } x_j\in A \text{ and } x_j^{-1}\not\in A,\\
x_j\mapsto x_ja^{-1} & \text{if } x_j\not\in A \text{ and } x_j^{-1}\in A,\\
x_j\mapsto ax_ja^{-1} & \text{if } x_j,x_j^{-1}\in A.\\
\end{cases}$
\end{center}
\end{mydef}

\begin{mydef}\label{defWhiteheadgraph}
Let $w$ be a cyclically reduced word. The \emph{Whitehead graph} of  $w$ is the undirected graph defined as follows:
\begin{enumerate}
    \item The alphabet $\Sigma$ is the vertex set.
    \item For every pair of consecutive letters in $w$, include an undirected edge between the inverse of the first letter and the second. Additionally these is also an edge connecting the inverse of the last letter of $w$ to the first letter of $w$.
\end{enumerate} 
\end{mydef}

Notice that, $w$ being cyclically reduced, in the Whitehead graph of $w$, we never have any edge connecting a vertex to itself.

\begin{mydef}\label{defcutvertex}
Let $w$ be a cyclically reduced word. A vertex $a$ in the Whitehead graph of $w$ is called a \emph{cut vertex} if it is non-isolated and at least one of the following two configurations take place:
\begin{enumerate}
    \item \label{def:cut1} The connected component of $a$ doesn't contain $a^{-1}$.
    \item \label{def:cut2}The connected component of $a$ becomes disconnected if we remove $a$.
\end{enumerate}
\end{mydef}

The notion of a cut vertex used here is a variant of the classical graph-theoretic definition, adapted to the structure of the Whitehead graph. In standard graph theory, a cut vertex is one whose removal increases the number of connected components, this is reflected in the second condition of our definition. The first condition arises from a phenomenon noted in \cite{stong1997diskbusting} and formalized in Proposition 2.2(b) of \cite{stallings1999whitehead}: a cyclically reduced word may contribute edges to multiple components of the Whitehead graph, and when a generator $a$ and its inverse $a^{-1}$ appear in different components, this indicates the possibility of a length-reducing Whitehead automorphism. Including this scenario in our definition ensures that the graph-theoretic notion of cut vertex accurately captures all algebraically meaningful situations relevant to automorphic reducibility. We recall Whitehead's theorem below. 

\begin{mythm}\label{cutvertex}
Let $w$ be a cyclically reduced word, which is primitive but not a single letter. Then the Whitehead graph of $w$ contains a cut vertex. Further, there is a Whitehead automorphism $\varphi$ such that the cyclic length of $\varphi(w)$ is strictly smaller than the cyclic length of $w$.
\end{mythm}

We state a refinement of Whitehead's theorem. See Theorem 3.7 in \cite{ascari2021fine}. 
 
\begin{mythm}\label{finesse}
The automorphism in \textnormal{Theorem \ref{cutvertex}} can be chosen in such a way that every $a$ or $a^{-1}$ letter, which is added when we apply $\varphi$ to $w$ letter by letter, immediately cancels (in the cyclic reduction process).
\end{mythm}

\subsection{Whitehead’s algorithm for free factors}

We now recall the definition of Whitehead graph for subgroups from \cite{ascari2021fine}.

\begin{mydef}\label{Letters}
Let $G = (V,E,\mu)$ be a labeled graph, $\Sigma$ be the labeling set and let $v\in G$ be a vertex. Define the set $L_v$ as the set of labels of outgoing edges at $v$. More precisely, we have $x_i\in L_v$ if and only if $G$ contains an edge $e \in E$ such that $\mu(e) = x_i$ and $o(e) =v$, and $x_i^{-1}\in L_v$ if and only if $G$ contains an edge $e \in E$ such that $\mu(e) = x_i$ and $t(e) =v$.
\end{mydef}

\begin{mydef}
Let $G$ be a labeled graph and $\Sigma$ be the labeling set. The \emph{Whitehead graph of} $G$ is the undirected graph defined as follows:
\begin{enumerate}
    \item The alphabet $\Sigma$ is the vertex set.
    \item For every vertex $v\in G$ and for every pair $x_i,x_j\in L_v$ of distinct letters at $v$, there is an undirected edge between $x_i$ and $x_j$ in the Whitehead graph.
\end{enumerate}
\end{mydef}

Notice that the Whitehead graph contains a complete subgraph with vertex set $L_v$ for every vertex $v\in G$; moreover, the Whitehead graph is exactly the union of these complete subgraphs. Taking $G = \Delta_H$, we obtain the Whitehead graph of a nontrivial finitely generated subgroup $H \leq F_m.$ When the subgroup $H$ is generated by a single word, i.e., $H = \langle w \rangle$, this graph coincides with the Whitehead graph of the cyclic reduction of $w$. In this sense, the notion of the Whitehead graph of $H$ generalizes the earlier Definition \ref{defWhiteheadgraph}. The notion of a cut vertex can likewise be defined for the Whitehead graph of a subgroup, exactly as in Definition \ref{defcutvertex}.

{\myexa \label{example:free} Recall the subgroup $H = \langle yx^{-1}, yzy^{-1}zt\rangle \leq F_4$ as discussed in Example \ref{example:H}. Notice that $H$ is a free factor of $F_4$. 
In the extended core graph $\widehat{\Delta}_H$, the set of labels $L_v$ associated with the outgoing edges at each vertex $v$ is as follows: $L_1 = \{x,y,t^{-1}\}$, $L_2 = \{x^{-1},y^{-1},z\}$, $L_3 = \{y^{-1},z^{-1}\}$, $L_4 = \{y,z\}$, and $L_5 = \{z^{-1},t\}$. 
\begin{figure}[!htb]
        \centering
        \includegraphics[width=\textwidth]{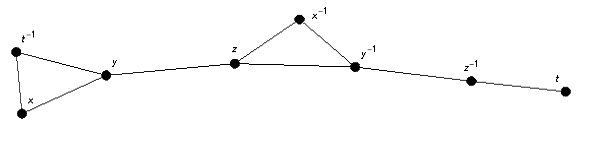}
        \caption{The Whitehead graph of $H$}
        \label{fig:W_H}
\end{figure}
Notice that $L_1\cap L_2 = \emptyset,$ where the vertices $1$ and $2$ represent the origin and the terminal vertex of the edge $e$ in $\widehat{\Delta}_H$ with label $y$, corresponding to the cut vertex $y$ in Figure \ref{fig:W_H}.}\\

We recall the analogues of Theorems \ref{cutvertex} and \ref{finesse} for free factors. See Theorem 5.5, 5.6 and 5.7 in \cite{ascari2021fine}.
\begin{mythm}\label{cutvertex2}
Let $H\leq F_m$ be a free factor, and suppose $\Delta_{H}$ has more than one vertex. Then the Whitehead graph of $H$ contains a cut vertex.
\end{mythm}

\begin{mythm}\label{Whitehead2}
Let $H\leq F_m$ be a free factor, and suppose the Whitehead graph of $H$ contains a cut vertex. Then there is a Whitehead automorphism $\varphi$ such that $\Delta_{\varphi(H)}$ has strictly fewer vertices and strictly fewer edges than $\Delta_{H}$.
\end{mythm}

We now recall Theorem 5.7 from \cite{ascari2021fine} that uses the notion $L_v$ introduced in Definition \ref{Letters}.

\begin{mythm}\label{finesse2}
The automorphism $\varphi=(A,a)$ in \textnormal{Theorem \ref{Whitehead2}} can be chosen in such a way that, at each vertex $v$ of $\Delta_H$, exactly one of the following configurations takes place:
\begin{enumerate}
    \item $L_v\cap A = \emptyset$.
    \item $L_v\subseteq A$.
    \item \label{thm:case3} $a\in L_v$ and $L_v\subseteq A\cup\{a\}$.
\end{enumerate}
\end{mythm}
The following statements are immediate consequences of Theorems \ref{Whitehead2} and \ref{finesse2}.
{\mycor \label{cor:ver_core_phi} Let $H\leq F_m$ be a free factor, and suppose the Whitehead graph of $H$ contains a cut vertex. Also, let $S_o \subset \widehat{V}$ be the set containing the vertices of $\Delta_H$ that fall in the case \ref{thm:case3} of \textnormal{Theorem \ref{finesse2}}. Then the set $S_o$ is non-empty. Additionally, the set $\widehat{V}\setminus S_o$ is the vertex set of $\Delta_{\varphi(H)},$ where $\widehat{V}$ is the vertex set of $\Delta_H$.}

\begin{myrem}\label{rem:collaps_sets}
Let $a$ be the cut vertex, and let $S_o \subseteq \widehat{V}$ be as defined in the Corollary \ref{cor:ver_core_phi}. Define the following sets:
\begin{itemize}
    \item  $E_o = \left\{ e \in E_{\Delta_H} \mid \mu(e) = a \text{ and } o(e) = v \in S_o \right\}$, the set of directed edges labeled $a$ whose origin lies in $S_o$.    
    \item $S_t = \bigl\{ t(e) \in \widehat{V} \mid e \in E_o  \bigr\}$, the set of terminal vertices of edges in $E_o$.
    \item $E_t = \{ \bar{e} \in \widehat{E} \mid e \in E_o \}$, the set of formal inverses of edges in \( E_o \).
\end{itemize}
Then by Theorem \ref{Whitehead2}, the sets $E_o, S_t$ and $E_t$ are non-empty. Moreover, each of the sets $S_o, E_o, S_t, E_t$ have the same cardinality:
$|S_o| = |E_o| = |S_t| = |E_t|.$
\end{myrem}

{\mypro \label{pro:labels} Let $a$ be the cut vertex. Let $e \in E_o$ such that $o(e) = v$ and $t(e) = v'$. Then the following statements are true.
\begin{enumerate}
    \item \label{pro_case1} $L_v \cap L_{v'} = \emptyset,$ where $L_v$ and $ L_{v'}$ denote the sets of labels of outgoing edges at vertices $v$ and $v'$ in $\widehat{\Delta}_H.$
    \item \label{pro_case2} Let $x \in L_v$, $y \in L_{v'}$. Then $L_{(v,x^{-1})} = L_v\setminus\{x\} \textrm{ and } L_{(v',y^{-1})} = L_{v'}\setminus\{y\} .$
\end{enumerate} where $L_{(v,x^{-1})}$, and $L_{(v',y^{-1})}$ denote the sets of labels of outgoing edges at $(v,x^{-1})$ and $(v',y^{-1})$ in $\B_H$, respectively.}
\begin{proof}
    Let $e \in E_o$ such that the label $\mu(e) = a$ is a cut-vertex of the Whitehead graph of $H$, with $o(e) = v$ and $t(e) = v'$. 
    \begin{enumerate}
    \item By Definition \ref{defcutvertex}, if configuration \ref{def:cut1} occurs, the fact that $a \in L_v$ and $a^{-1} \in L_{v'}$ ensures that both $L_v$ and $L_{v'}$ belong to disjoint connected components. If configuration \ref{def:cut1} does not occur, then configuration \ref{def:cut2} must take place. In this case, removing the cut-vertex $a$ disconnects the component $G$ of the Whitehead graph of $H$ containing $a$. We are then left with at least two nonempty connected components, $G_1$ and $G_2$, and at least one of these components (let's say $G_1$) does not contain $a^{-1}$. As the letters in $L_v$ represent vertices of a complete subgraph of the Whitehead graph of $H$, the set $L_v$ must be contained in $G_1$. This yields the case. 
    \item This statement follows from the Equation \eqref{eqn:B_H_transition} of the definition of $\B_H$. This completes the proof.
\end{enumerate}
\end{proof}

{\myexa \label{example:freeH} From the Whitehead graph of $H = \langle yx^{-1}, yzy^{-1}zt\rangle$ in Figure \ref{fig:W_H} (see Example \ref{example:free}), the vertex $y$ is a cut vertex whose removal disconnects the graph. Other cut vertices include $z, y^{-1}$ and $z^{-1}$, but we choose $y$ to define the Whitehead automorphism $\varphi$. Let $A = \{x, t^{-1}\}$ be the set of vertices in the component not containing $y^{-1}$. This yields $\varphi = (\{x, t^{-1}\}, y)$, where $\varphi(x) = yx$, $\varphi(y) = y$, $\varphi(z) = z$, and $\varphi(t) = ty^{-1}$. The automorphism satisfies the trichotomy of Theorem \ref{finesse2} for $\Delta_H$, and we get $\varphi(H) = \langle x^{-1}, zy^{-1}zt\rangle$. Observe that $S_o = \{1\},E_o = \{e\}, S_t = \{2\} \textrm{ and } E_t = \{\bar{e}\},$ where $e$ is the dashed edge in Figure \ref{fig:coreH}.}

\section{The Cogrowth Inequality}
In this section, our main objective is to extend Theorem \ref{Whitehead2} to the automaton $\B_H$ and derive $\B_{\varphi(H)}$. Here, $\B_H$ and $\B_{\varphi(H)}$ denote ergodic automata that recognize the irreducible languages $L_H$ and $L_{\varphi(H)}$ associated with $H$ and $\varphi(H)$, respectively. Additionally, we will describe a method for obtaining $M_1$ from $M$, where $M_1$ and $M$ are the adjacency matrices of the transition graphs of $\B_{\varphi(H)}$ and $\B_H$, respectively. Through this method, we establish a strict cogrowth inequality, showing $\lambda_1 > \lambda$, where $\lambda_1$ and $\lambda$ denote the Perron-Frobenius eigenvalues of $M_1$ and $M$, respectively. 
\subsection{An Extension of Ascari's Result}
\begin{figure}[!htb]
\begin{subfigure}[b]{0.5\textwidth}
        \centering
          \begin{tikzpicture}[scale=.65]
    \tikzstyle{knode}=[circle,draw=black,thick,text
width = 30 pt,align=center, gray, inner sep=1pt]
\tikzstyle{rnode}=[circle,draw=black,thick,text
width = 25 pt,align=center,inner sep=1pt]
\node (q1) at (-3,2) [rnode] {$o(e_1)$};
\node (q12) at (-3,0)[] {$\vdots$};
\node (q2) at (-3,-2) [rnode] {$o(e_d)$};
\node (q0) at (0,0) [knode]{$(s,a^{\e})$} ;
\node (q3) at (3,-2) [rnode] {$t(\ee_{d'})$};
\node (q34) at (3,0)[] {$\vdots$};
\node (q4) at (3,2) [rnode] {$t(\ee_1)$};
\draw[->,gray, dashed] (q1) to node[above=2pt] {$a^{\e}$} (q0);
\draw[->,gray, dashed] (q2) to node[above=2pt] {$a^{\e}$} (q0);
\draw[->,black, thick] (q0) to node[above=2pt] {$x_{d'}^{\e_{d'}}$} (q3);
\draw[->,black, thick] (q0) to node[above=2pt] {$x_{1}^{\e_{1}}$} (q4);
    \end{tikzpicture}
    \caption{Adjacency in $\B_H$}
    \label{fig:adja_H}
\end{subfigure}
       \hfill
\begin{subfigure}[b]{0.45\textwidth}
        \centering
    \begin{tikzpicture}[scale=.65]
    \tikzstyle{knode}=[circle,draw=black,thick,text
width = 30 pt,align=center,inner sep=1pt]
\tikzstyle{rnode}=[circle,draw=black,thick,text
width = 25 pt,align=center,inner sep=1pt]
\node (q1) at (-3,2) [rnode] {$o(e_1)$};
\node (q12) at (-3,0)[] {$\vdots$};
\node (q2) at (-3,-2) [rnode] {$o(e_d)$};
\node (q3) at (3,-2) [rnode] {$t(\ee_{d'})$};
\node (q34) at (3,0)[] {$\vdots$};
\node (q4) at (3,2) [rnode] {$t(\ee_1)$};
\draw[->] (q1) to (-1.5,1) node[above] {$x_{d'}^{\e_{d'}}$} to (q3);
\draw[->] (q1) to node[above=2pt] {$x_{1}^{\e_{1}}$} (q4);
\draw[->] (q2) to node[above=2pt] {$x_{d'}^{\e_{d'}}$} (q3);
\draw[->] (q2) to (-1.5,-1) node[above] {$x_{1}^{\e_{1}}$} to (q4);
    \end{tikzpicture}
    \caption{Adjacency in $\B_{\varphi(H)}$}
    \label{fig:adja_phiH}
    \end{subfigure}
    \caption{ }
    \label{fig:schreier_generator}
\end{figure}

    Our goal is to derive $\B = \bigl(Q, \Sigma, \delta, I, F\bigr)$ from $\B_H = \bigl(Q_H, \Sigma, \delta_H, I_H, F_H\bigr)$, using a Whitehead automorphism $\varphi=(A,a)$ given by Theorem \ref{Whitehead2}. We construct a non-empty subset $S$ of $Q_H$ using non-empty sets $S_o$ and $S_t$, defined as:
    $$S = \bigl\{(v, a^{-1}) , (v',a) \in Q_H | v \in S_o \textnormal{ and } v' \in S_t\bigr\}.$$
    For any $(s,a^{\e}) \in S$, we denote the $in$ and $out$ degree of the vertex $(s,a^{\e})$ in $\B_H$ as $d$ and $d'$, respectively. These degrees, $d$ and $d'$, are dependent on $(s,a^{\e})$. i.e. There are $d$ edges in $\B_H$, denoted as $\displaystyle e_i, i = 1,\cdots,d$, terminating at $(s,a^{\e})$ with a common label $a^{\e}$ (See gray dashed edges in Figure \ref{fig:adja_H}). Similarly, there are $d'$ edges in $\B_H$, denoted as $\displaystyle \ee_j, j = 1,\cdots,d'$, originating from $(s,a^{\e})$ with labels $x_j^{\e_j}$, where $x_j^{\e_j} \neq a^{-\e}$ (See black edges in Figure \ref{fig:adja_H}). 
    
    Recall from the Lemma 5.16 of \cite{ascari2021fine} that the core $\Delta_{\varphi(H)}$ can be obtained by collapsing all edges $e \in E_o$, where $e$ connects $o(e) = v \in S_o$ to $t(e) = v'\in S_t$ and has label $a$. Additionally, collapsing an edge $e$ in $\Delta_H$ results in the collapse of two edges, namely, $e$ and $\bar{e}$ in $\widehat{\Delta}_H$, where $\bar{e}$ connects $o(\bar{e}) = v'$ to $t(\bar{e}) = v$ and has label $a^{-1}$. Consequently, the extended core $\widehat{\Delta}_{\varphi(H)}$ can be obtained from the extended core $\widehat{\Delta}_H$, by collapsing all edges $e \in E_o$ and $\bar{e} \in E_t$. Equivalently, we obtain the ergodic automaton $\B$ from $\B_H$ by collapsing all the $d$ edges in $\B_H$, namely $\displaystyle e_i, i = 1,\cdots,d$ with label $a^{\e}$, associated with each $(s,a^{\e}) \in S$. See Figure \ref{fig:adja_phiH}. This process of collapsing is described in the theorem below.
    
\begin{mythm}\label{thm:B_from_BH}  
Let $H\leq F_m$ be a non-cyclic free factor, and suppose the Whitehead graph of $H$ contains a cut vertex. Then there is a Whitehead automorphism $\varphi$ such that the ergodic automaton $\B_{\varphi(H)}$ that recognizes $L_{\varphi(H)}$ can be obtained from $\B_H$ by collapsing certain edges of $\B_H$.
\end{mythm}
\begin{proof}
    Given $(s,a^{\e}) \in S$ and the edge $e_i, i = 1, \cdots, d$, there are $d'$ two-length paths from each $o(e_i)$ to $\displaystyle t(\ee_j), j = 1,\cdots,d'$ passing through the vertex $(s,a^{\e})$ in $\B_H$. Observe that $o(e_i), ~\displaystyle t(\ee_j) \notin S$, for $ j = 1,\cdots,d'$. These paths have labels $a^{\e}x_j^{\e_j}$, where $x_j^{\e_j} \neq a^{-\e}$. After collapsing the edge $e_i$, these $d'$ two-length paths become $d'$ one-length paths (or $d'$ edges) from $o(e_i)$ to $t(\ee_j), j = 1,\cdots,d'$ with labels $x_j^{\e_j}$ in $\B$. 
    Call $Q = Q_H\setminus S$, i.e., for each $i = 1, \cdots, d$, we have 
    \begin{equation}
    \delta(o(e_i), x_j^{\e_j}) = t(\ee_j), j = 1,\cdots,d',\label{eqn:B2_transition}    
    \end{equation}
    while $\delta(o(e_i), a^{\e})$ does not exist in $\B$. It's worth noting that the collapsing process does not affect edges $e$ in $\B_H$ whose origin $o(e)$ and terminal $t(e)$ vertices are not both in $S$. Consequently, the maps $\delta$ and $\delta_H$ are identical for these edges. Thus, we have determined the set of states $Q$, the transition map $\delta$, and the alphabet $\Sigma$ of $\B$. We are now left with determining the set of initial and final states, namely, $I$ and $F$, of $\B$.
    
    If $s$ from the given state $(s,a^{\e}) \in S$, is not the root vertex of $\Delta_H$, then the initial and final vertices of $\B_H$ and $\B$ coincide. However, if $s$ is the root vertex, then $(s,a^{\e}) \in I_H$. Additionally, according to Remark 5.9 of \cite{ascari2021fine}, $(s,a^{-\e}) \notin S$ which implies $(s,a^{-\e}) \in I.$ For all $x^{\e} \in \Sigma \setminus \{a, a^{-1}\}$, if $(s,x^{\e}) \in I_H$ then $(s,x^{\e}) \in I.$ To determine the set of initial states of $\B$, we remove $(s,a^{\e})$ from $I_H$ and add the $d$ states $o(e_i)$, for $i = 1,\cdots,d$. Consequently, the sets of initial and final states of $\B$ are given by:
    \begin{equation}
        I = F = \bigl(I_H \setminus \{(s,a^{\e})\} \bigr)\cup \bigl\{o(e_i) | i = 1,\cdots,d\bigr\}. \label{eqn:B_initial_set2}
    \end{equation}
    This indicates that the set of initial states $I$ and final states $F$ of $\B$ depend on the vertex $s$ of the given state $(s,a^{\e}).$ In either case, the set of states is $Q$ and the transition map $\delta$ is as follows:
    \begin{gather}
    \delta\left((v,x_i^{\e}),x_j^{\e'}\right) = \delta_H\left((v,x_i^{\e}),x_j^{\e'}\right) = (vx_j^{\e'},x_j^{\e'}), \textrm{ if } (v,x_i^{\e}), (vx_j^{\e'},x_j^{\e'}) \notin S   \label{eqn:B1_transition}
    \end{gather}
    In addition to the Equation \eqref{eqn:B1_transition}, new edges that occur due to collapsing are given in Equation \eqref{eqn:B2_transition}. 
    If $s$ from the given state $(s,a^{\e}) \in S$, is not the root vertex of $\Delta_H$, then $I = F = I_H = F_H.$ Whereas if $s$ is the root vertex, then the sets of initial and final states of $\B$ are given in Equation \eqref{eqn:B_initial_set2} and thus, we obtain the finite automaton $\B.$
    
    Recall from Theorem \ref{thm:B_m_properties} that $\B_H$ is deterministic and has homogeneous ambiguity $|I_H| - 1$. Since $(s, a^{\epsilon}) \in S$, we must have either $(s, a^{\epsilon}) = (v, a^{-1})$ or $(s, a^{\epsilon}) = (v', a)$. Without loss of generality, assume $(s, a^{\epsilon}) = (v, a^{-1})$. By construction, each origin $o(e_i)$ is of the form $(v', x_j^{-1})$ for some $x_j \in L_{v'} \setminus \{a^{-1}\}$. From Proposition \ref{pro:labels}, the sets of labels of outgoing edges from each $o(e_i)$ are disjoint from those of $(v, a^{-1})$. Therefore, when we collapse the states $o(e_i)$ into $(v, a^{-1})$, the automaton $\B$ remains deterministic and since $\B$ is obtained from $\B_H$ by collapsing edges, it inherits homogeneous ambiguity $|I| - 1$.

    To show that $\B$ is the minimal ergodic automaton that recognizes $L_{\varphi(H)}$, let $w \in L(\B)$. Then there are exactly $|I|-1$ distinct admissible paths with label $w$ in $\B$. Notice from the construction of $\B_H$ and hence from the construction of $\B$ that these $|I|-1$ paths are distinct only at the initial state. Suppose none of the vertices of these paths belong to the set $S$. Then $w \in L_H$ such that $\varphi(w) = w$ and therefore $w \in L_{\varphi(H)}.$ Suppose some vertices of these paths do belong to the set $S$. Then there is a word $w' \in L_H$ such that $\varphi(w') = w \in \varphi(H)$. $w \in L(\B)$ implies that $w$ is reduced, hence $w \in L_{\varphi(H)}$. Thus $L(\B) \subseteq L_{\varphi(H)}$.
    
    To show that $L_{\varphi(H)} \subseteq L(\B)$, we first write $\B_{\varphi(H)}$ using $\widehat{\Delta}_{\varphi(H)}$ as:
    $$\B_{\varphi(H)} = \bigl(Q_{\varphi(H)}, \Sigma, \delta_{\varphi(H)}, I_{\varphi(H)}, F_{\varphi(H)}\bigr),$$
    The description of $Q_{\varphi(H)}, \Sigma, \delta_{\varphi(H)}, I_{\varphi(H)},$ and $ F_{\varphi(H)}$ in terms of the core $\widehat{\Delta}_{\varphi(H)}$ follows analogously from that of $\B_H$ given in Section \ref{sec:def_B_H}, and is therefore omitted.
    
    Let $w \in L_{\varphi(H)}$. Then $w \in \varphi(H)$ and $\varphi^{-1}(w) \in H$. If $\varphi^{-1}(w) = w$ then we are through. If not, then some vertices along each of the $|I_H| - 1$ admissible paths in $\B_H$ with the label $\varphi^{-1}(w)$ belong to the set $S.$ The collapsing process implies that $\varphi(\varphi^{-1}(w)) \in L(\B)$ implies that $w \in L(\B).$ It follows that $L(\B) = L_{\varphi(H)}$. Observe that $|Q| = |Q_{\varphi(H)}|.$ Recall from Theorem \ref{thm:B_m_properties} that $\B_{\varphi(H)}$ is minimal ergodic automaton such that $L(\B_{\varphi(H)}) = L_{\varphi(H)}.$ Therefore, the transition diagrams of the automata $\B_{\varphi(H)}$ and $\B$ are isomorphic, which completes the proof.
    \end{proof}

    {\myexa \label{exa:autB_H} Recall $H = \langle yx^{-1}, yzy^{-1}zt\rangle \leq F_4$ from Example \ref{example:freeH}. In the transition diagram of $\B_H$, the elements of the set $S = \{(2,y),(1,y^{-1})\}$ are denoted by states $11$ and $12$, respectively. See Figure \ref{fig:BandPhiB} for the transition diagrams of $\B$ and $\B_{\varphi(H)}$.
    \begin{figure}[!htb]
        \centering
        \begin{subfigure}[b]{\textwidth}
             \centering
             \includegraphics[width=\textwidth]{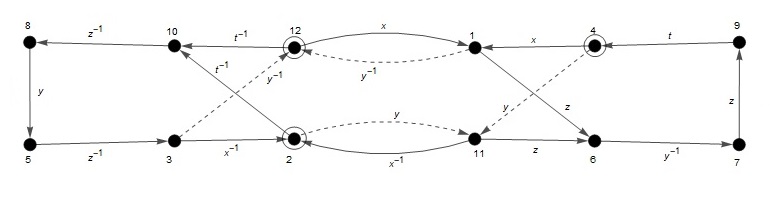}
             \caption{The automaton $\B_H$}
             \label{fig:B_H}
        \end{subfigure}
        \newline
        \begin{subfigure}[b]{\textwidth}
            \centering
            \includegraphics[width=\textwidth]{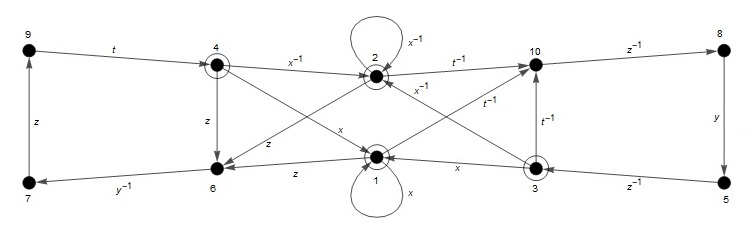}
            \caption{The automaton $\B_{\varphi(H)}$}
            \label{fig:B_PhiH}
        \end{subfigure}
        \caption{}
        \label{fig:BandPhiB}
    \end{figure}
     In these figures, dashed edges represent those to be collapsed, while circled vertices denote the initial and final states. The details on the state enumeration are provided in the subsequent section.}
    \subsection{The state enumeration}\label{ssec:SE}
     To facilitate a precise comparison between the automata $\B_H$ and $\B_{\varphi(H)}$, we adopt a modified enumeration of states. While the particular enumeration is not mathematically intrinsic, it plays an essential technical role in aligning the corresponding states of the two automata. This alignment, in turn, enables us to describe the transformation from the adjacency matrix $M$ of $\B_H$ to that of $\B_{\varphi(H)}$ via explicit and elementary operations.
     
     In the vertex enumeration of $\Delta_H$, vertices are represented as $v_k$, where $k = 1,\cdots,|\widehat{V}|$, with $v_1$ being the root vertex. After collapsing each edge $e$ from the set $E_o$, we identify the vertex $o(e)$ with $t(e)$ and label it as $t(e)$. Thus, the vertex enumeration of $\Delta_{\varphi(H)}$ is derived from that of $\Delta_H$ by removing the origins of edges in the set $E_o.$
    
     In the state enumeration of $\B_H$ as described in \cite{DGS2021}, states are represented as $(v_k,a_i^{\e})$, where $k$ ranges from $1$ to $|\widehat{V}|$, $i$ ranges from $1$ to $rank(H)$, and $\e$ takes values in $\{1,-1\}$. We call the state enumeration provided in \cite{DGS2021} as the old state enumeration (OSE). For our purposes, we derive a new state enumeration (NSE) by implementing two modifications on OSE as follows:
     Recall that for every $e \in E_o$, there is a pair of $S$-states, namely, $\{(o(e),a^{-1}), (t(e),a)\} \subset S \subset Q_H.$
     \begin{enumerate}
        \item Partitioning: For every $e \in E_o$, remove the pair of $S$-states from OSE and append them at the end. This results in a partition of $Q_H$, with the non-$S$-states preceding the $S$-states. 
        \item Grouping and Reordering: For every $e \in E_o$ among the remaining non-$S$-states, group together those with the first component as either $o(e)$ or $t(e)$ in the enumeration. Within this group, while ignoring the first component, reorder them with respect to their second component.  
     \end{enumerate}
     The immediate implication of these modifications is stated as follows:
     {\myrem \label{rem:OSE_varphiH} In the second modification, for each $e \in E_o$, we formed a group of states by disregarding their first components. We then relabeled these first components as $t(e)$. Consequently, the resulting non-$S$-state partition of NSE of $\B_H$ corresponds to the OSE of $\B_{\varphi(H)}$.}
     
     {\myexa \label{exa:enumeration} The OSE of $\B_H$ in Example \ref{exa:autB_H} is given by $(1,x^{-1})$, $(1,y^{-1})$, $(1,t)$, $(2,x)$, $(2,y)$, $(2,z^{-1})$, $(3,y)$, $(3,z)$, $(4,y^{-1})$, $(4,z^{-1})$, $(5,z)$ and $(5,t^{-1})$. Recall that $S = \{(2,y),(1,y^{-1})\}$ i.e. $E_o = \{e\}$.
     After implementing 1st modification, we obtain the non $S$-states as $(1,x^{-1})$, $(1,t)$, $(2,x)$, $(2,z^{-1})$, $(3,y)$, $(3,z)$, $(4,y^{-1})$, $(4,z^{-1})$, $(5,z)$, $(5,t^{-1})$ and the $S$-states as $(2,y)$ and $(1,y^{-1})$. As $o(e)$ and $t(e)$ in this example are $1$ and $2$, respectively. Considering the edge $e$, we have 4 states, namely, $(1,x^{-1})$, $(1,t)$, $(2,x)$ and $(2,z^{-1})$ among the non $S$-states. Applying the 2nd modification, we reorder these 4 states $(2,x)$, $(1,x^{-1})$, $(2,z^{-1})$ and $(1,t)$. Thus, our NSE of $\B_H$ is $(2,x)$, $(1,x^{-1})$, $(2,z^{-1})$, $(1,t)$, $(3,y)$, $(3,z)$, $(4,y^{-1})$, $(4,z^{-1})$, $(5,z)$, $(5,t^{-1})$, $(2,y)$ and $(1,y^{-1})$. By Remark \ref{rem:OSE_varphiH}, the OSE of $\B_{\varphi(H)}$ (See Figure \ref{fig:corePhiH}) is $(2,x)$, $(2,x^{-1})$,  $(2,z^{-1})$, $(2,t)$, $(3,y)$, $(3,z)$, $(4,y^{-1})$, $(4,z^{-1})$, $(5,z)$ and $(5,t^{-1})$.} 
     
    The new enumeration naturally induces a block structure on the matrix $M$, as described in the following corollary.
    
     \begin{mycor}\label{cor:M_partition} The matrix $M$ can be decomposed as follows:
            \begin{equation}
                M = \left(
            \begin{array}{cc}
             M' & U \\
             Z & O \\
            \end{array}
            \right), \label{eqn:decomposition_M}
            \end{equation}
            where $M'$, $U$, $Z$, and $O$ are sub-matrices with dimensions $|Q|\times|Q|$, $|Q|\times|S|$, $|S|\times|Q|$, and $|S|\times|S|$, respectively. This decomposition satisfies the following properties:
            \begin{enumerate}
                \item \label{cor:C}The rows of the matrix $U$ contain either all zeros or exactly one non-zero entry equal to 1.
                \item \label{cor:O}The matrix $O$ is a zero matrix.
            \end{enumerate}
    \end{mycor}
    \begin{proof}
        The decomposition of $M$ follows from NSE. 
            \begin{enumerate}
                \item Let the matrix $U$ has a row (referred to as the $(q,x^{\e})$-row) containing at least two entries equal to 1, i.e., there exists states $(q,x^{\e}) \in Q$, $(s,a^{\e})$, $(s',a^{-\e}) \in S$, and corresponding edges in $\B_H$:
                $$ \delta_H\big((q,x^{\e}),a^{\e}\big)= (s,a^{\e}) \textrm{ and } \delta_H\big((q,x^{\e}),a^{-\e}\big)= (s',a^{-\e}).$$
                From Equation \eqref{eqn:B_H_transition}, we have:
                $$ \delta_{\widehat{\Delta}_H}\left(q,a^{\e}\right)= s \textrm{ and } \delta_{\widehat{\Delta}_H}\left(q,a^{-\e}\right)= s'.$$
                Since $(q,x^{\e}) \in Q = Q_H\setminus S$, it implies that either $s \in S_o$ or $s' \in S_o.$ Without loss of generality, let $s \in S_o$. Then there exists an edge $e$ in $\Delta_H$ such that $o(e) = q$ and $\mu(e) = a^{\e}$. This implies that the set of labels of outgoing edges at $q$ in the graph $\widehat{\Delta}_H$ contains both $a^{\pm 1}$. However, from Remark 5.9 of \cite{ascari2021fine}, this contradicts the fact that $s \in S_o$.
                \item This statement follows directly from the definition of the set $S$.
            \end{enumerate}
    \end{proof}

    We now describe a procedure to transform the adjacency matrix $M$ of $\B_H$ into a new matrix $M_1$, which corresponds to the automaton $\B$ whose transition diagram is isomorphic to that of $\B_{\varphi(H)}$. This transformation uses the block decomposition of $M$ described in Corollary~\ref{cor:M_partition} and corresponds to the construction of automaton in Theorem~\ref{thm:B_from_BH}. The steps are as follows: 
    \begin{itemize} 
    \item[1.]  For each $(s,a^{\e}) \in S$, identify the associated row and column in $M$, denoted $R_{(s,a^{\e})}$ and $C_{(s,a^{\e})}$, respectively.
    \item[2.] In the column $C_{(s,a^{\e})}$, locate all nonzero entries (which are necessarily equal to 1). Each such entry corresponds to a row $R_{o(e_i)}$ in $M$, for $i = 1, \cdots, d$, where $d$ is the number of such entries. 
    \item[3.] For each $i$, replace the row $R_{o(e_i)}$ by the sum $R_{o(e_i)} + R_{(s,a^{\e})}$.
    \item[4.] After performing these additions, delete the row $R_{(s,a^{\e})}$ and the column $C_{(s,a^{\e})}$ from $M$. 
    \item[5.] Repeat this process for every $(s,a^{\e}) \in S$. The resulting matrix is $M_1$, a square matrix of size $|Q|\times |Q|$.
    \end{itemize} 

    \begin{mycor}\label{cor:M1_from_M}
    Let $M$ be the adjacency matrix of the automaton $\B_H$, and let $M_1$ be the $|Q|\times |Q|$ matrix obtained by applying the transformation procedure described above. Then $M_1$ is the adjacency matrix of the transition diagram of $\B$, which is isomorphic to $\B_{\varphi(H)}$.
    \end{mycor}
    \begin{proof}
    By Theorem~\ref{thm:B_from_BH}, the automaton $\B$ is isomorphic to $\B_{\varphi(H)}$, and its structure is obtained from $\B_H$ by eliminating each state $(s,a^{\e}) \in S$ and replacing it with direct transitions from each state $o(e_i)$, for $i = 1,\cdots,d$, to each state $t(\ee_j)$, for $j = 1,\cdots,d'$, where $e_i$ and $\ee_j$ are edges in the automaton $\B_H$ such that $t(e_i) = o(\ee_j) = (s,a^{\e})$. 

    To reflect this in the adjacency matrix, the row $R_{(s,a^{\e})}$ (which encodes transitions from $(s,a^{\e})$) is added to each row $R_{o(e_i)}$, thereby redirecting transitions from $o(e_i)$ directly to each $t(\ee_j)$. After this redistribution, the row and column corresponding to $(s,a^{\e})$ are removed.

    Repeating this process for every element of $S$ yields a matrix $M_1$ of size $|Q|\times |Q|$, which captures the transitions of the automaton $\B$.
    \end{proof}
    
    \subsection{The Inequality}
    We recall the following standard theorem (see Theorem 1.5 of \cite{bestvina1992train} or Theorem 1.6 on page 23 and Exercises 2.1 on page 39 of \cite{seneta2006non}), which holds significant importance in our analysis. It's important to clarify that by `vector' we refer to column vectors in the context of this discussion.
    \begin{mythm}\label{thm:Perron-Frobenius} (Perron-Frobenius). Suppose that $N$ is an irreducible, non-negative integral matrix. Then there is a unique positive eigenvector $\overrightarrow{w}$ of norm one for $N$, and its associated eigenvalue satisfies $\eta \geq 1$. If $\eta = 1$, then $N$ is a transitive permutation matrix. Moreover if $\overrightarrow{u}$ is a positive vector and $\beta > 0$ satisfies 
    $(N\overrightarrow{u})_i \leq \beta\overrightarrow{u}_i$ for each $i$ and 
    $(N\overrightarrow{u})_j < \beta\overrightarrow{u}_j$ for some $j$, then $\eta < \beta.$
    \end{mythm}
    
    Recall that $M = (m_{ij})$ and $M_1$ are the adjacency matrices of the deterministic and ergodic automata $\B_H$ and $\B$, respectively.  As a result, they are non-negative, integral (with entries either $0$ or $1$), and irreducible. Since $H$ is a non-cyclic free factor of $F_m$, and $\varphi(H)$ is as well, this implies that $\lambda > 1$ and $\lambda_1 > 1$. With all the required machinery in place, we are now ready to prove that the Perron-Frobenius eigenvalue of $M$ is strictly smaller than that of $M_1$.

    \begin{mythm} \label{thm:main_lambda<lambda1}
        Let $\lambda$ and $\lambda_1$ be the Perron-Frobenious eigenvalue of $M$ and $M_1$, respectively, where $M$ and $M_1$ are as described in Corollary \ref{cor:M1_from_M}. 
        Then $ \lambda < \lambda_1.$
    \end{mythm}
    \begin{proof} 
    Recall that $Q_H = Q \cup S$. We use the notation $(q, x_k^{\e})$ for states in $Q$ and $(s, a^{\e})$ for states in $S$. Let $R_{(v, x_k^{\e})}$ and $C_{(v, x_k^{\e})}$ denote the rows and columns of the matrix $M$ corresponding to the state $(v, x_k^{\e}) \in Q_H$, and let $R'_{(q, x_k^{\e})}$ and $C'_{(q, x_k^{\e})}$ denote the rows and columns of the matrix $M_1$ corresponding to the state $(q, x_k^{\e}) \in Q$.

    Choose a positive vector $\overrightarrow{v}$ in $\mathbb{R}^{|Q|}$ so that $M_1\overrightarrow{v} = \lambda_1\overrightarrow{v}$, and let $\overrightarrow{u}$ be the vector in $\mathbb{R}^{|Q|+|S|}$ defined by $u_{(q,x_k^{\e})} = v_{(q,x_k^{\e})}$ for ${(q,x_k^{\e})} \in Q$. We will determine the remaining $|S|$ components of $\overrightarrow{u}$ based on the block decomposition of $M$ described in Corollary  \ref{cor:M_partition}.
    
    By Theorem \ref{thm:B_from_BH}, in the column $C_{(s,a^{\e})}$ in $M$, there are $d$ entries with the value $1$ and the rest are $0$. The rows in which these $d$ entries appear are $R_{o(e_i)}$, respectively, where $e_i, i = 1,\cdots,d$, and are the edges in the transition graph of $\B_H$ terminating at $(s,a^{\e})$ with label $a^{\e}.$ For each $o(e_i) \in Q, i = 1,\cdots,d$, we use Corollary \ref{cor:M_partition} to write 
     $$(M \overrightarrow{u})_{o(e_i)} = b_i + u_{(s,a^{\e})},$$  
     where $b_i = \displaystyle \sum_{(q,x_k^{\e'}) \in Q} m_{o(e_i)(q,x_k^{\e'})} u_{(q,x_k^{\e'})}.$ Also, for $(s,a^{\e}) \in S$, we write 
     $$(M \overrightarrow{u})_{(s,a^{\e})} = \sum_{(q,x_k^{\e'}) \in Q} m_{(s,a^{\e})(q,x_k^{\e'})} u_{(q,x_k^{\e'})} = b_{(s,a^{\e})}.$$
     
    By Corollary \ref{cor:M1_from_M}, the corresponding rows $R'_{o(e_i)}$ of $M_1$ can be obtained from $M$ by applying row addition 
    $$R_{o(e_i)} + R_{(s,a^{\e})} \rightarrow R_{o(e_i)},$$ 
    where $i = 1, \dots, d$, and removing rows $R_{(s,a^{\e})}$ and columns $C_{(s,a^{\e})}$ from the resultant, where $(s,a^{\e}) \in S$. 

    Using the above relationships between rows $R_{o(e_i)}, R_{(s,a^{\e})},$ and the column $C_{(s,a^{\e})}$ of $M$, and the corresponding rows $R'_{o(e_i)}$ of $M_1$, for $i = 1,\cdots,d$, we obtain
    $$(M \overrightarrow{u})_{o(e_i)} + (M \overrightarrow{u})_{(s,a^{\e})} - u_{(s,a^{\e})} = (M_1 \overrightarrow{v})_{o(e_i)} = \lambda_1 u_{o(e_i)}.$$

    Substituting the expressions for $(M \overrightarrow{u})_{o(e_i)}$ and $(M \overrightarrow{u})_{(s,a^{\e})}$, we obtain
    $$b_{(s,a^{\e})} = \lambda_1 u_{o(e_i)} - b_i.$$
    Since $b_i > 0$ and $\lambda_1 u_{o(e_i)} > b_i$, it follows that $0 < b_{(s,a^{\e})} < \lambda_1 u_{o(e_i)}.$
    To ensure that $M \overrightarrow{u} \leq \lambda_1 \overrightarrow{u}$, we require
    $$(M \overrightarrow{u})_{o(e_i)} = b_i + u_{(s,a^{\e})} \leq \lambda_1 u_{o(e_i)} \quad \text{and} \quad (M \overrightarrow{u})_{(s,a^{\e})} = b_{(s,a^{\e})} \leq \lambda_1 u_{(s,a^{\e})}.$$
    Thus, we need to choose $u_{(s,a^{\e})}$ such that $$\frac{b_{(s,a^{\e})}}{\lambda_1} \leq u_{(s,a^{\e})} \leq b_{(s,a^{\e})}.$$
    This interval is non-empty, since $b_{(s,a^{\e})} > 0$ and $\lambda_1 > 1$. Therefore, for each $(s,a^{\e}) \in S$, we choose 
    $u_{(s,a^{\e})} = \frac{b_{(s,a^{\e})}}{\lambda_1}$. Then, 
    $$(M \overrightarrow{u})_{(s,a^{\e})} = b_{(s,a^{\e})} = \lambda_1 u_{(s,a^{\e})}, \quad (M \overrightarrow{u})_{o(e_i)} = b_i + u_{(s,a^{\e})} < \lambda_1 u_{o(e_i)}.$$
    Hence, for this choice, we have $(M \overrightarrow{u})_{(v, x^{\e})} \leq \lambda_1 u_{(v, x^{\e})}$ for all $_{(v, x^{\e})} \in Q_H$, with strict inequality for $(v, x^{\e}) = o(e_i),$ where $ i = 1,\cdots, d$. 
    By Theorem \ref{thm:Perron-Frobenius}, we conclude that $\lambda < \lambda_1.$ This completes the proof.
\end{proof}

The following Corollary follows from Theorem \ref{thm-on-computing-entropy} and Theorem \ref{thm:main_lambda<lambda1}. 
\begin{mycor}\label{cor:Entropy_inequality} 
    Let $\varphi$ be the automorphism given by Theorem \ref{thm:B_from_BH}. Then $$ent(L_H) < ent(L_{\varphi(H)}),$$ where $ent(L_H) = \log \lambda$ and $ent(L_{\varphi(H)}) = \log \lambda_1$.
\end{mycor}
The subsequent Corollary follows from Equation \eqref{eqn:cogrowth_entropy} and Theorem \ref{thm:main_lambda<lambda1}. 
\begin{mycor} \label{cor:cogrowthrates_inequality}
Let $\varphi$ be the automorphism given by Theorem \ref{thm:B_from_BH}. Then $$\alpha_H < \alpha_{\varphi(H)},$$ where 
$\alpha_H$ and $\alpha_{\varphi(H)}$ represent the cogrowths of $H$ and $\varphi(H)$, respectively.
\end{mycor}
\section{Example} \label{sec:exa}
Recall our free factor $H = \langle yx^{-1}, yzy^{-1}zt\rangle$ of the free group $F\langle x,y,z,t \rangle$. Following Theorems \ref{thm:B_m_properties} and \ref{thm:B_from_BH}, as well as Corollary \ref{cor:M1_from_M}, we construct the automata $\B_H$ and $\B_{\varphi(H)}$, illustrated in Figures \ref{fig:B_H} and \ref{fig:B_PhiH} respectively. Using Corollary \ref{cor:M1_from_M}, one can obtain $M_1$ from the below given matrix $M$. 
\NiceMatrixOptions%
 {code-for-first-row = \color{gray}\scriptstyle \rotate \text{},
 code-for-last-col = \color{gray}\scriptstyle }
 \[M = \begin{pNiceMatrix}[first-row,vlines = 11,last-col=13]
 (2,x) & (1,x^{-1}) & (2,z^{-1}) & (1,t) & (3,y) & (3,z) & (4,y^{-1}) & (4,z^{-1}) & (5,z) & (5,t^{-1}) & (2,y) & (1,y^{-1})\\
 0 & 0 & 0 & 0 & 0 & 1 & 0 & 0 & 0 & 0 & 0 & 1 & (2,x)\\
 0 & 0 & 0 & 0 & 0 & 0 & 0 & 0 & 0 & 1 & 1 & 0 & (1,x^{-1})\\
 0 & 1 & 0 & 0 & 0 & 0 & 0 & 0 & 0 & 0 & 0 & 1 & (2,z^{-1})\\
 1 & 0 & 0 & 0 & 0 & 0 & 0 & 0 & 0 & 0 & 1 & 0 & (1,t)\\
 0 & 0 & 1 & 0 & 0 & 0 & 0 & 0 & 0 & 0 & 0 & 0 & (3,y)\\
 0 & 0 & 0 & 0 & 0 & 0 & 1 & 0 & 0 & 0 & 0 & 0 & (3,z)\\
 0 & 0 & 0 & 0 & 0 & 0 & 0 & 0 & 1 & 0 & 0 & 0 & (4,y^{-1})\\
 0 & 0 & 0 & 0 & 1 & 0 & 0 & 0 & 0 & 0 & 0 & 0 & (4,z^{-1})\\
 0 & 0 & 0 & 1 & 0 & 0 & 0 & 0 & 0 & 0 & 0 & 0 & (5,z)\\
 0 & 0 & 0 & 0 & 0 & 0 & 0 & 1 & 0 & 0 & 0 & 0 & (5,t^{-1})\\
 \hline
 0 & 1 & 0 & 0 & 0 & 1 & 0 & 0 & 0 & 0 & 0 & 0 & (2,y)\\
 1 & 0 & 0 & 0 & 0 & 0 & 0 & 0 & 0 & 1 & 0 & 0 & (1,y^{-1})
\end{pNiceMatrix}\]
    The Perron Frobenius eigenvalue of $M_1$ is $\lambda_1 = 1.64$ the transpose $\overrightarrow{v}^t$ of the associated eigenvector $\overrightarrow{v}$ of $M_1$ is 
    $$(\overrightarrow{v})^t = \left( 3.12 , 3.12 , 4.41 , 4.41 , 2.69 , 1 , 1.64 , 1.64 , 2.69 , 1 \right).$$    
    From Theorem \ref{thm:main_lambda<lambda1}, we define the vector $\overrightarrow{u}$ by setting $u_i = v_i$ for $i = 1,\cdots,10$ and choosing 
    $v_{11} = v_{12}  = 2.51$ since $b_{(2,y)} = b_{(1,y^{-1})} = 4.12.$ 
    This gives $$(M\overrightarrow{u})_i < (\lambda_1\overrightarrow{u})_i$$ for $i = 1,2,3,4$ and the equality $(M\overrightarrow{u})_i = (\lambda_1\overrightarrow{u})_i$ for $i = 5,\cdots,12$. Thus, the Perron Frobenius eigenvalue $\lambda$ of $M$, is strictly less than $\lambda_1$, and in this case, $\lambda = 1.45$.
\section{Open Problem}
The results established in this article suggest several directions for further exploration. We pose the following open questions:

\begin{myque}
Can similar cogrowth inequalities be formulated for conjugacy classes of subgroups, using restricted core graphs (obtained by removing the possible hair ending at the root in the core graph) as the natural invariant? 
\end{myque}

\begin{myque}
Is it possible to extend Whitehead-type reductions to tuples of subgroups? If so, how do these reductions interact with product automorphisms and their associated graphs? What new insights can be gained by studying cogrowth in this extended context?
\end{myque}

\begin{myque}
Can the techniques discussed in this article be extended to arbitrary finitely generated subgroups of $F_m$, or are there inherent limitations that restrict the results to non-cyclic free factors? In particular, can tools such as the Marshall Hall Theorem or properties of finite-index subgroups be used to generalize the inequality between the Perron–Frobenius eigenvalues of the associated matrices $M$ and $M_1$?
\end{myque}

\begin{myque}
Is it possible to prove Corollary \ref{cor:Entropy_inequality} using methods inspired by Gromov, thus avoiding the use of Perron–Frobenius theory? How would this alternative approach compare with the current proof in terms of technique and applicability, as discussed in \cite{CECCHERINISILBERSTEIN200393}?
\end{myque}

Building on the Whitehead maximal cogrowth problem mentioned in the introduction, another open question arises regarding the effectiveness of certain automorphisms. The cogrowth inequality suggests that among all images of a non-cyclic free factor $H$ under automorphisms of $F_m$, those subgroups whose core graphs have a single vertex achieve the highest cogrowth (or entropy).

\begin{myque} Given a finitely generated subgroup $H \leq F_m$, can one determine an automorphism $\varphi \in Aut(F_m)$ such that the cogrowth $\alpha_{\varphi(H)}$ is maximized over the automorphic orbit of $H$? Is there an effective method or algorithm to find such an automorphism? 

\end{myque}
\section{Acknowledgments}
The author would like to thank the anonymous referees for their valuable suggestions, including insightful open problems, which significantly improved the exposition and depth of the paper.
\bibliographystyle{plain} 

\end{document}